\documentclass[twoside,11pt]{article}
\usepackage{graphicx,amscd,amsfonts,amsmath,amsthm,amssymb,latexsym,amsfonts,color,extsizes,multirow,algorithm,algorithmic,mathtools}
\usepackage{epsfig,float,epstopdf,array,bm,cite,cases,array,multirow,relsize,graphicx,float,booktabs}
\usepackage{pgfplots}
\pgfplotsset{compat=newest}
\usepackage{geometry}
\usepackage{booktabs}
\usepackage{geometry}
\usepackage{array}
\usepackage{siunitx}
\usepackage{caption}
\usepackage{subcaption}
\usepackage[bookmarksnumbered,colorlinks,plainpages]{hyperref}
\hypersetup{citecolor=blue}

\newtheorem{theorem}{\bf Theorem}
\newtheorem{lemma}{\bf Lemma}
\newtheorem{proposition}{\bf Proposition}

\newtheorem{example}{Example}
\newtheorem{remark}{Remark}

\def\diag{\textrm{diag}}

\begin{document}
	
	\title{\bf\LARGE{
A block preconditioner with two parameters for double saddle point systems in liquid crystal director modeling
	}}
	\author{\bf Mehdi Makhdomi$^{\dag}$, Davod Khojasteh Salkuyeh$^{\ddag}$\thanks {\noindent Corresponding author.}~, Morad Ahmadnasab$^{\dag}$ 
		\\	
		{ $^{\dag}$\textit{Department of Mathematics, University of Kurdistan, P.O. Box 416, Sanandaj,  Iran}} \\
		{ $^{\ddag}$\textit{Faculty of Mathematical Sciences, University of Guilan, Rasht, Iran}} \\
		{ Emails: mehdi.ap.math@gmail.com,   khojasteh@guilan.ac.ir, m.ahmadnasab@uok.ac.ir}}
	
	\date{}
	\maketitle
	\vspace{-0.5cm}
	\noindent\hrulefill\\
	{\bf Abstract.} 
	This paper introduces a novel block preconditioner for a class of double saddle point systems arising from the finite element discretization of liquid crystal director modeling. We first establish sufficient conditions to guarantee the convergence of the underlying iterative scheme. A spectral analysis of the preconditioned matrix is then performed, leading to an effective strategy for parameter selection. The new preconditioner is used to accelerate the convergence of the flexible version of GMRES. Experimental results confirm the strong performance of the method.
	
	\noindent\hrulefill\\
	
	\pagestyle{myheadings}\markboth{M. Makhdomi, D. K. Salkuyeh, M. Ahmadnasab}{ }
	\thispagestyle{empty}

\section{Introduction}
Liquid crystals (LCs) are a class of materials that exhibit mesophases, which are states of matter intermediate between isotropic liquids and crystalline solids. They combine key characteristics of both phases: LCs flow like conventional liquids, while their constituent molecules possess long-range orientational order analogous to that observed in crystalline solids; see~\cite{Ramage2013, Gennes1993, Stewart2004} for further details.

The orientational configuration of a liquid crystal director can be described by a function that minimizes a free energy functional of the form
\begin{align}\label{eqq1}
\mathcal{F}[u, v, w, U] = \frac{1}{2} \int_{0}^{1} \left[ (u_z^2 + v_z^2 + w_z^2) - \hat{\alpha}^2 (\hat{\beta} + w^2) U_z^2 \right] \, dz,
\end{align}
where \( u, v, w, \) and \( U \) are functions of the spatial variable \( z \in [0, 1] \), subject to appropriate boundary conditions at the endpoints. Here, \( u_z = \frac{du}{dz} \) (and similarly for \( v_z, w_z, \) and \( U_z \)). The parameters \(\hat {\alpha} \) and \( \hat{\beta} \) are given positive constants.

The free energy functional~\eqref{eqq1} is discretized using a uniform piecewise linear finite element method with mesh size \( h = 1/(l+1) \), with as \(l + 1\) the number of cells. The resulting discrete system is then minimized subject to a set of constraints imposed on the discrete solution vectors \( u, v, \) and \( w \).


By applying the method of Lagrange multipliers in conjunction with a Newton-type iterative scheme to this constrained minimization problem, one obtains a double saddle point system of the form
\begin{align}\label{org}
\tilde{\mathcal{A}} \mathbf{u} \equiv
\begin{pmatrix}
A & B^{T} & C^{T} \\
B & 0 & 0 \\
C & 0 & -D
\end{pmatrix}
\begin{pmatrix}
x_{1} \\
x_{2} \\
x_{3}
\end{pmatrix}
=
\begin{pmatrix}
f_{1} \\
f_{2} \\
f_{3}
\end{pmatrix}
= \mathbf{d},
\end{align}
where \( A \in \mathbb{R}^{n \times n} \) and \( D \in \mathbb{R}^{p \times p} \) are symmetric positive definite matrices (SPD)~\cite{Ramage2013}. The matrix \( B \in \mathbb{R}^{m \times n} \) is of full row rank, while \( C \in \mathbb{R}^{p \times n} \) is a rank-deficient matrix with \( n > m + p \). Moreover, \( BB^{T} \) is diagonal, and, in particular, \( BB^{T} = I_m \) if the unit vector constraints are satisfied exactly, with \( I_m \in \mathbb{R}^{m \times m} \) being the identity matrix.
 The vectors \( f_{1} \in \mathbb{R}^{n} \), \( f_{2} \in \mathbb{R}^{m} \), and \( f_{3} \in \mathbb{R}^{p} \) are given data, and \( \mathbf{u} = (x_{1}; x_{2}; x_{3})\) represents the unknown vector to be determined.
 
  Due to the following proposition, it follows that the system~\eqref{org} possesses a unique solution.
\begin{proposition}
	\label{prop1}
	Let \( A \) and \( D \) be SPD matrices. The block matrix \( \tilde{\mathcal{A}} \) defined in~\eqref{org} is
	 invertible if and only if \( B \) has the full row  rank~\cite{Beik2018}.
\end{proposition}

Double saddle point problems of the form~\eqref{org} arise in numerous areas of scientific computing and engineering, including constrained optimization, mixed finite element formulations of fluid flow problems~\cite{Boffi2013, Ramage2013}, among others.

It is obvious that the coefficient matrix $\mathcal{\tilde{A}}$ of the double saddle point 
problem~\eqref{org} can be regarded as a $2\times 2$ block matrix in two ways, namely
\[
\mathcal{\hat{A}} = \left( \begin{array}{cc|c} 
A & B^{T} & C^{T} \\  
B & 0 & 0 \\ 
\hline 
C & 0 & -D
\end{array} \right),
\]
or
\[
\mathcal{\hat{A}} = \left( \begin{array}{c|cc}
A & B^{T} & C^{T} \\ 
\hline 
B & 0 & 0 \\ 
C & 0 & -D
\end{array} \right).
\]
which may both be viewed as generalized saddle point matrices.

A wide range of iterative schemes has been developed for generalized saddle point systems, including Uzawa-type iterations~\cite{Baiz2008, Zho2009}, SOR-type methods~\cite{Feng2018, Golub2001}, Hermitian and skew-Hermitian splitting (HSS) approaches~\cite{ChenF2018, ChenF2020, Liao2019, Zhang2018}, and shift-splitting (SS) techniques~\cite{Cao2014, Salkuyeh2015}, among others. These iteration methods can
not only be used as stationary iteration solvers, but they can also be served as preconditioners
for Krylov subspace methods~\cite{Cao2019, Salkuyeh2016} leading to a more efficacious class of solvers.

A variety of preconditioning strategies have been developed to efficiently address double saddle point (DSP) systems that arise in models of liquid crystal directors. In this section, we summarize several of the most notable approaches.

In \cite{BeikBenzi}, Beik and Benzi introduced an approach for solving the double saddle point system \eqref{org} using GMRES method by block triangular preconditioners having one of the following structures
\[
\mathcal{\tilde{P}}_T = 
\begin{pmatrix}
A & B^T & C^T \\
0 & -BA^{-1}B^T & 0 \\
0 & 0 & -(D + CA^{-1}C^T)
\end{pmatrix},
\]
and
\[
\mathcal{\hat{P}}_T = 
\begin{pmatrix}
A & B^T & C^T \\
0 & -BA^{-1}B^T & -BA^{-1}C^T \\
0 & 0 & -(D + CA^{-1}C^T)
\end{pmatrix}.
\]
Beik and Benzi \cite{Beik2018} introduced two preconditioners for DSP systems: the block diagonal (BD) and block triangular (BT) variants. These are defined as
\begin{align}\label{BD_BT}
\mathcal{P}_{D} &=
\begin{pmatrix}
A & B^{T} & 0 \\
\gamma B & 0 & 0 \\
0 & 0 & -D
\end{pmatrix}, &
\mathcal{P}_{T} &=
\begin{pmatrix}
A & B^{T} & C^{T} \\
\gamma B & 0 & 0 \\
0 & 0 & -D
\end{pmatrix},
\end{align}
where \( \gamma \) is a specified constant. Benzi and Beik~\cite{Benzi2019} proposed a class of Uzawa-like methods specifically designed for double saddle point systems. Their work introduced a generalized (block) Gauss-Seidel (GGS) scheme and a generalized (block) successive overrelaxation (GSOR) method, providing a comprehensive framework for solving such problems. For a detailed exposition of these methods and their theoretical properties, the reader is referred to \cite{Benzi2019}.

 In~\cite{ChenRen2023}, Chen and Ren proposed an iterative framework for solving DSP systems, accompanied by a novel preconditioning approach expressed as
\begin{align}\label{TP}
\mathcal{P} =
\begin{pmatrix}
D & -C & 0 \\
0 & S + A & B^{T} \\
0 & -B & 0
\end{pmatrix}.
\end{align}
Here, \(S\) denotes a given symmetric positive semidefinite matrix. Building upon these developments, Zhu et al.~\cite{Zhuw} recently formulated a two-parameter block triangular (TPBT) preconditioner of the form
\begin{align}\label{TPBT}
\mathcal{P}_{\mathrm{TPBT}} =
\begin{pmatrix}
A & B^{T} & C^{T} \\
B & -\alpha I & 0 \\
0 & 0 & -(D + \beta C C^{T})
\end{pmatrix},
\end{align}
with positive parameters \( \alpha \) and \( \beta \). The authors also provided efficient computational formulas to determine quasi-optimal choices of these parameters, enhancing the preconditioner’s performance. Most recently, Zhang~\cite{Zhang25} introduced the triangular upper block preconditioner, defined as
\begin{align}\label{PBUT}
\mathcal{P}_{\mathrm{BUT}} =
\begin{pmatrix}
A & B^{T} & C^{T} \\
B & -\alpha I & 0 \\
0 & 0 & -(D + \beta I)
\end{pmatrix},
\end{align}
where \( \alpha > 0 \) and \( \beta > 0 \) are given parameters. This formulation provides another efficient strategy for preconditioning DSP systems in liquid crystal modeling. Through a series of numerical experiments, Zhang demonstrated that the triangular upper block preconditioner she introduced outperforms the preconditioners presented in~\eqref{BD_BT} and~\eqref{TPBT}.

In this work, we aim to introduce a novel two-parameter preconditioner for double saddle point systems arising in liquid crystal director models, designed to further enhance the convergence rate of iterative
solvers while maintaining computational efficiency. Its formulation builds upon existing block triangular approaches, incorporating tunable parameters that can be optimized to achieve near-optimal performance across a range of problem instances.

Throughout this manuscript, we employ the following notation. For any square matrix $A$,  
$\rho(A)$ denotes its spectral radius, $\operatorname{tr}(A)$ its trace, and $\Lambda(A)$ its spectrum.  
The largest and smallest eigenvalues of a symmetric matrix $A$ are written as $\lambda_{\max}(A)$ and 
$\lambda_{\min}(A)$, respectively. For a general matrix $A$, $\|A\|_{2}$ represents the matrix 
$2$-norm, while $\|A\|_{F}$ denotes the Frobenius norm; the corresponding largest and smallest 
singular values are given by $\sigma_{\max}(A)$ and $\sigma_{\min}(A)$. We denote the null space of the matrix $A$ by $\mathcal{N}(A)$.

The notation $A \succ 0$ (resp., $A \succeq 0$) indicates that $A$ is a SPD (resp., symmetric positive semi-definite) matrix. Furthermore, for two matrices $A$ and $B$, the relation $A \succ B$ (resp., $A \succeq B$) denotes that $A - B \succ 0$ (resp., $A - B \succeq 0$).

For any complex vector $x \in \mathbb{C}^{n}$, its conjugate transpose is denoted by $x^{*}$.  
When combining vectors $x$, $y$, and $z$, we adopt the \textsc{Matlab} notation $(x;y;z)$ to 
indicate the column vector obtained by stacking them vertically, i.e., $(x^{T}, y^{T}, z^{T})^{T}$, or 
$(x^{*}, y^{*}, z^{*})^{*}$ in the complex case. Parameters equipped with the subscript “est,” such 
as $\alpha_{\text{est}}$, refer to estimated values of the corresponding quantities.

The remainder of this paper is organized as follows. Section~\ref{sec2} introduces the proposed preconditioner and analyzes the convergence of the associated iterative method. Implementation aspects of the preconditioner are also discussed. In Section~\ref{sec3}, the spectral properties of the preconditioned matrix are examined. The strategy for selecting the preconditioner parameters is presented in Section~\ref{sec4}. Section \ref{sec5} is devoted to numerical experiments that assess the performance, efficiency, and
robustness of the proposed approach. Finally, Section~\ref{sec6} offers concluding remarks.

\section{The new block preconditioner}\label{sec2}

In this section, we introduce a block preconditioner with two parameters (BPT) for solving the double saddle point systems given in equation~\eqref{org}. We start by considering an alternative splitting of the coefficient matrix $\mathcal{\tilde{A}}$, expressed as
\begin{equation}\label{split}
\mathcal{\tilde{A}} = \mathcal{Q}_{\mathrm{BPT}} - \mathcal{L}_{\mathrm{BPT}},
\end{equation}
where the matrices $ \mathcal{Q}_{\mathrm{BPT}} $ and $ \mathcal{L}_{\mathrm{BPT}} $ are defined as follows
\begin{equation}
\mathcal{Q}_{\mathrm{BPT}} = 
\begin{pmatrix}
A + \alpha B^{T}B & B^T & 0 \\
B & -\beta I & 0 \\
C & 0 & -D
\end{pmatrix}, \quad \mathcal{L}_{\mathrm{BPT}} = 
\begin{pmatrix}
\alpha B^{T}B & 0 & -C^{T} \\
0 & -\beta I & 0 \\
0 & 0 & 0
\end{pmatrix},
\end{equation}
in which $\alpha$ and $\beta$ are given positive numbers.
 It is straightforward to verify that the matrix $\mathcal{Q}_{\mathrm{BPT}}$ is nonsingular.
 The iterative scheme corresponding to this splitting \eqref{split} takes the form

\begin{align}\label{iter}
\mathbf{u}^{(k+1)} = \mathcal{G}_{\alpha,\beta}\mathbf{u}^{(k)} + \mathbf{g}, \quad k = 0, 1, 2, \ldots,
\end{align}
where $\mathbf{u}^{(0)}$ denotes the initial approximation, $\mathcal{G}_{\alpha,\beta} = \mathcal{Q}_{\mathrm{BPT}}^{-1}\mathcal{L}_{\mathrm{BPT}}$ is the iteration matrix, and $\mathbf{g} = \mathcal{Q}_{\mathrm{BPT}}^{-1}\mathbf{d}$. The iterative process defined in equation~\eqref{iter} converges to the exact solution of system~\eqref{org} for any arbitrary initial vector if and only if the spectral radius condition $\rho(\mathcal{G}_{\alpha,\beta}) < 1 $
is satisfied~\cite{SaadBook}. To analyze the convergence properties of the proposed iterative scheme for the double saddle point system, we first recall a fundamental lemma concerning the maximum modulus of the roots of a quadratic polynomial. 
\begin{lemma}[{\cite{Young1971}}]\label{l1}
	Consider the quadratic equation $y^2 - qy + p = 0$, where $p$ and $q$ are real numbers.
	Both roots of the equation have a modulus less than one if and only if  $ |p| < 1$ and $|q| < 1 + p$.
\end{lemma}
\begin{theorem}\label{th1}
	Let $A \in \mathbb{R}^{n \times n}$ and $D \in \mathbb{R}^{p \times p}$ be SPD matrices, $B \in \mathbb{R}^{m \times n}$ be a matrix of full row rank, and $C \in \mathbb{R}^{p \times n}$ be rank-deficient and  $\alpha > 0$ and $\beta > 0$ . Define
	\[
	\mathcal{D} = \left\{\, x \in \mathbb{R}^{n} : (x; y; z) \text{ is an eigenvector of } \mathcal{G}_{\alpha,\beta} \text{ with } \|x\|_2 = 1 \,\right\}.
	\]
	Then, the iterative method~\eqref{iter} converges to the unique solution of the system $\mathcal{\tilde{A}}\mathbf{u} = \mathbf{d}$ for any initial guess if and only if
	\begin{equation}\label{ConvCond}
	c-a < \gamma b,
	\end{equation}
	for every
	\begin{equation}\label{aandb}
	a = x^{\ast} Ax > 0, 
	\quad 
	b = x^{\ast} B^{T} B x  \ge0 \quad \text{and} \quad c=x^{\ast} C^{T}D^{-1}C x \ge 0,
	\end{equation}
	where $x \in \mathcal{D}$, $\alpha $ , $ \beta>0 $ and $ \gamma=2\alpha+\frac{1}{2\beta}$. 
\end{theorem}

\begin{proof}
	Let $(\lambda, \mathbf{u})$ be an eigenpair of the iteration matrix $\mathcal{G}_{\alpha,\beta}$, where $\mathbf{u} = (x; y; z)$ with 
	$x \in \mathbb{R}^n$, $y \in \mathbb{R}^m$, and $z \in \mathbb{R}^p$.  Then, the eigenvalue equation
	\begin{equation}\label{eigen1}
	\mathcal{G}_{\alpha,\beta}\mathbf{u} = \lambda \mathbf{u}
	\end{equation}
	is equivalent to
	\begin{align}
	\alpha B^{T}Bx -C^{T}z &=\lambda(Ax+\alpha B^{T}Bx+B^{T}y), \label{eq2} \\ 
	-\beta y  
	&= \lambda(Bx-\beta y), \label{eq3}\\
	\lambda(Cx-Dz) &=0. \label{eq4}
	\end{align}	
	
	For $\lambda = 0$, the result follows trivially. Hence, we consider the case $\lambda = 1$.  
	From~\eqref{eigen1}, it follows that $\mathcal{\tilde{A}}\mathbf{u} = 0$. Since $\mathcal{\tilde{A}}$ is nonsingular, we deduce that $ \mathbf{u}= 0$,  
	 which contradicts the definition of an eigenvector.
	 
	 	We claim that $x \neq 0$. Next, suppose $x = 0$. From \eqref{eq4} and the fact that $\lambda \neq 1$, we obtain $z = 0$. From \eqref{eq3}, we get $y = 0$. However, this contradicts the assumption that $\mathbf{u} = (x;\, y;\, z)$ is a nonzero eigenvector. Therefore, we must have $x \neq 0$.
	
	Thus, we may assume that $\lambda \ne 0, 1$ and $x \ne 0$. From~\eqref{eq3} and~\eqref{eq4}, we obtain
	\begin{align}
	y &= \frac{\lambda}{\beta(\lambda - 1)} Bx, \label{eq8} \\
	z &= D^{-1} Cx. \label{eq9}
	\end{align}
	Substituting these expressions into~\eqref{eq2} yields
	\begin{equation}\label{eq10}
	\alpha B^{T} Bx - C^{T} D^{-1} Cx 
	= \lambda A x + \alpha \lambda B^{T} Bx + \frac{\lambda^{2}}{\beta(\lambda - 1)} B^{T} Bx.
	\end{equation}
	
	Without loss of generality, assume that $\|x\|_{2} = 1$.  
	Premultiplying~\eqref{eq10} by $x^{\ast}$ and multiplying both sides by $\beta(\lambda - 1)$, we obtain
	\begin{equation}\label{eq12}
	\alpha \beta b (\lambda - 1) - \beta c (\lambda - 1)
	= \beta a \lambda (\lambda - 1) + \alpha \beta b \lambda (\lambda - 1) + b \lambda^{2},
	\end{equation}
	where $a$, $b$, and $c$ are defined in~\eqref{aandb}.  
	After rearranging terms, this leads to the quadratic equation
	\begin{equation}\label{eq1.2}
	\lambda^{2}
	- \frac{a\beta + 2\alpha\beta b - \beta c}{a\beta + \alpha\beta b + b}\lambda
	+ \frac{\alpha\beta b - \beta c}{a\beta + \alpha\beta b + b} = 0.
	\end{equation}
	
	By Lemma~\ref{l1}, $|\lambda| < 1$ if and only if the following inequalities hold:
	\begin{align}
	\left| \frac{\alpha\beta b - \beta c}{a\beta + \alpha\beta b + b} \right| &< 1, \label{eq13} \\
	\left| \frac{a\beta + 2\alpha\beta b - \beta c}{a\beta + \alpha\beta b + b} \right|
	&< 1 + \frac{\alpha\beta b - \beta c}{a\beta + \alpha\beta b + b}. \label{eq14}
	\end{align}
	It is straightforward to verify that inequalities~\eqref{eq13} and~\eqref{eq14} are, respectively, equivalent to
	\begin{align}\label{eq15}
	c - a < \frac{2\alpha\beta + 1}{\beta}\, b,
	\end{align}
	and
	\begin{align}\label{eq16}
	c - a < \frac{4\alpha\beta + 1}{2\beta}\, b.
	\end{align}
	Since $\alpha, \beta, a > 0$, $c \ge 0$, and $b \ge 0$, inequalities~\eqref{eq15} and~\eqref{eq16} together imply
	\[
	c - a < \min\!\left\{ \frac{2\alpha\beta + 1}{\beta}\, b,\, \frac{4\alpha\beta + 1}{2\beta}\, b \right\}
	= \left( 2\alpha + \tfrac{1}{2\beta} \right) b = \gamma b,
	\]
		 which completes the proof.
\end{proof}

\begin{lemma}[{\cite[Theorem 7.7.3]{Horn2013}}]
	Let $A, B \in \mathbb{R}^{n \times n}$ be symmetric matrices, with $A$ being positive definite and $B$ positive semidefinite.
	 Then:
	\begin{enumerate}
		\item $A \succeq B$ if and only if $\rho(A^{-1}B) \le 1$;
		\item $A \succ B$ if and only if $\rho(A^{-1}B) < 1$.
	\end{enumerate}
\end{lemma}\label{lm2}

\begin{remark}\label{coro2}
	Applying Lemma~\ref{lm2}, to the linear systems of the form~\eqref{eqq1} arising from liquid crystal modeling
	 and noting that $A \succ 0$ and $C^{T}D^{-1}C \succeq 0$, condition~\eqref{ConvCond} in Theorem~\ref{th1} implies that the BPT iteration method converges if
	\[
	\rho\!\left(A^{-1}C^{T}D^{-1}C\right) < 1,
	\]
	which is equivalent to
	\[
	A \succ C^{T}D^{-1}C.
	\]
Our numerical investigations in Example~\ref{ex1} confirm that the condition 
$
A \succ C^{T}D^{-1}C
$
is satisfied for problems of small to moderate dimensions. Moreover, the numerical evidence strongly suggests that this condition continues to hold for larger-scale problems as well. On the basis of this observation, we therefore conclude that our method converges unconditionally whenever \(\alpha > 0\) and \(\beta > 0\). This claim is further supported by the numerical results reported in Table~\ref{tab:eigenvalues_AB}, which clearly demonstrate and validate the expected convergence behavior.
\end{remark}

We conclude this section by discussing the application of the preconditioner 
$\mathcal{Q}_{\mathrm{BPT}}$ within Krylov subspace methods for solving the linear system 
$\mathcal{\tilde{A}}\mathbf{u} = \mathbf{d}$. 
At each iteration of a Krylov method such as GMRES or its flexible variant FGMRES, 
one must compute the action of the inverse of the preconditioner, which requires solving a linear system of the form
$$ \mathcal{Q}_{\text{BPT}}z=r,$$ 
where $ z=(z_{1};z_{2};z_{3})$ and  $r=(r_{1};r_{2};r_{3})$ with $z_{1} , r_{ 1} \in \mathbb{R}^{n}$ , $z_{2} , r_{2} \in \mathbb{R}^{m} $ and $ z_{3} , r_{3}\in \mathbb{R}^{p} $. 
We can state Algorithm 1 for solving  $ \mathcal{Q}_{\text{BPT}}z=r$.

\medskip
\noindent\textbf{Algorithm 1.} \textit{Computation of $z = \mathcal{Q}_{\mathrm{BPT}}^{-1}r$ }\\[2mm]
1. Solve $(A+(\alpha+\frac{1}{\beta})B^{T}B) z_{1} = r_{1}+\frac{1}{\beta}B^{T}r_2$ for $z_{1}$; \\ [1.5mm]
2. Compute $z_{2} =  \frac{1}{\beta}(B z_{1}-r_2)$; \\ [1.5mm]
3. Solve $D z_{3} =Cz_1-r_3$ for $z_{3}$;\\[1.5mm]
4. Set $z=(z_1;z_2;z_{3})$.
\medskip

From Algorithm~1, we see that two linear subsystems  
with the coefficient matrices 
$
A + (\alpha + \tfrac{1}{\beta}) B^{T} B
$
and
$
D,
$
should be solved in Steps~1 and~3, respectively. 
Since both matrices are SPD, they may be solved exactly using the Cholesky factorization or approximately using the Conjugate Gradient (CG) method. 
We remark that when these subsystems are solved inexactly, the use of FGMRES is required to ensure the robustness and convergence of the outer iteration.

\section{Spectral analysis of the preconditioned matrix}\label{sec3}

\begin{theorem}\label{th2}
	Suppose that the matrices \(A\),\(B\),\(C\), and \(D\) satisfy the assumptions of Theorem~\ref{th1}. Then the preconditioned matrix
	\(\mathcal{Q}_{\mathrm{BPT}}^{-1} \tilde{\mathcal{A}}\) has an eigenvalue equal to \(1\) with algebraic multiplicity at least \(p\).
	Moreover, the remaining eigenvalues of \(\mathcal{Q}_{\mathrm{BPT}}^{-1} \tilde{\mathcal{A}}\) coincide with the eigenvalues of the matrix

	\begin{equation}\label{Gamma}
	\Gamma =
	\begin{pmatrix}
	I + \bigl( C^{T} D^{-1} C - \alpha B^{T} B \bigr) M^{-1} &
	\dfrac{1}{\beta} \bigl( C^{T} D^{-1} C - \alpha B^{T} B \bigr) M^{-1} B^{T} \\[8pt]
	B M^{-1} &
	\dfrac{1}{\beta} B M^{-1} B^{T}
	\end{pmatrix},
	\end{equation}
	where
$	M = A + \bigl( \alpha + \tfrac{1}{\beta} \bigr) B^{T} B .
$
\end{theorem}

\begin{proof}
	A straightforward computation yields
	\begin{align}\label{Pmatrix}
	\tilde{\mathcal{A}} \mathcal{Q}_{\mathrm{BPT}}^{-1} =
	\begin{pmatrix}
	I + \left( C^{T} D^{-1} C - \alpha B^{T} B \right) M^{-1} &
	\dfrac{1}{\beta} \left( C^{T} D^{-1} C - \alpha B^{T} B \right) M^{-1} B^{T} &
	- C^{T} D^{-1} \\[6pt]
	B M^{-1} &
	\dfrac{1}{\beta} B M^{-1} B^{T} &
	0 \\[3pt]
	0 & 0 & I
	\end{pmatrix}.
	\end{align}
	Note that the matrix \(\tilde{\mathcal{A}} \mathcal{Q}_{\mathrm{BPT}}^{-1}\)
	is similar to \(\mathcal{Q}_{\mathrm{BPT}}^{-1} \tilde{\mathcal{A}}\); hence,
	they share the same eigenvalues. It follows immediately that
	\(\mathcal{Q}_{\mathrm{BPT}}^{-1} \tilde{\mathcal{A}}\) has an eigenvalue equal
	to \(1\) with algebraic multiplicity at least \(p\). The remaining eigenvalues
	are those of the \(2\times2\) block matrix \(\Gamma\) defined in~\eqref{Gamma}.  This completes the proof.
\end{proof}

\begin{theorem}\label{th3}
	Let \(A \in \mathbb{R}^{n \times n}\) and \(D \in \mathbb{R}^{p \times p}\) be SPD matrices.
	Let also \(B \in \mathbb{R}^{m \times n}\) has full row rank, and let \(C \in \mathbb{R}^{p \times n}\) be rank--deficient.
	Assume that \(\alpha > 0\) and \(\beta > 0\), 
	if the parameters \(\alpha\) and \(\beta\) satisfies
	\begin{align}\label{alphabeta}
	\frac{\alpha}{\beta} \le \frac{1}{4}\,\lambda_{\min}^{2}(A),
	\end{align}
	then all the eigenvalues of the preconditioned matrix
	\(\mathcal{Q}_{\mathrm{BPT}}^{-1}\tilde{\mathcal{A}}\) are real and positive.
	Moreover, its spectrum satisfies
	\[
	\Lambda\!\bigl(\mathcal{Q}_{\mathrm{BPT}}^{-1}\tilde{\mathcal{A}}\bigr)
	\subset
	\left(
	0,
	\;
	1+\frac{
		 \lambda_{\max}\!\left(C^{T} D^{-1} C\right)
	}{
		 \lambda_{\min}(A)
	}
	\right].
	\]
	
\end{theorem}
 
\begin{proof}
	Let $\mu$ be an eigenvalue of $\mathcal{Q}_{\mathrm{BPT}}^{-1} \tilde{\mathcal{A}}$ with the eigenvector 
	$\mathbf{u} = (x; y; z)$, where $x \in \mathbb{R}^{n}$, $y \in \mathbb{R}^{m}$, and $z \in \mathbb{R}^{p}$. 
Hence, $ \tilde{\mathcal{A}}\mathbf{u}=\mu\mathcal{Q}_{\mathrm{BPT}}\mathbf{u} $, which equivalent to
	\begin{align}
	A x + B^{T} y + C^{T} z &= \mu \big(A x + \alpha B^{T} B x + B^{T} y \big), \label{eqa} \\
	B x &= \mu (B x - \beta y), \label{eqb} \\
	C x - D z &= \mu (C x - D z). \label{eqc}
	\end{align}
	Since both matrices \(\mathcal{Q}_{\mathrm{BPT}}^{-1}\) and \(\tilde{\mathcal{A}}\) are nonsingular, we infer that $\mu \neq 0$. We consider two cases $ \mu=1 $ and $ \mu\neq 1 $.
	
	\medskip
	\noindent\textbf{Case 1: $\mu = 1$.}  
	Eq.~\eqref{eqb} gives $y = 0$. Substituting this into~\eqref{eqa} yields
	\begin{align}\label{bz}
	C^{T} z = \alpha B^{T} B x.
	\end{align}
	Multiplying both sides of Eq.~\eqref{bz} by $B$ and using the fact that in the liquid‐crystal setting $B B^{T} = I_m$ gives
	\begin{equation}\label{bxz}
	B(x-\frac{1}{\alpha}C^{T} z)=0.
	\end{equation}
Consequently, the eigenvector corresponding to the unit eigenvalue takes the form
\[
\mathbf{u} = \left( x;\, 0;\, z \right),
\]
for each $x, z$ satisfying $x - \frac{1}{\alpha}C^{T}z \in \mathcal{N}(B)$.

	\medskip
	\noindent\textbf{Case 2: $\mu \neq 1$.}  
	From~\eqref{eqb} and~\eqref{eqc} we obtain
	\[
	y = \frac{\mu - 1}{\beta \mu}\, Bx, 
	\qquad 
	z = D^{-1} C x,
	\]
	with $x \neq 0$ (see the proof of  Theorem~\ref{th1}). Without loss of generality assume that $\|x\|_2 = 1$. 
	Substituting these expressions into~\eqref{eqa} and premultiplying by $x^{\!*}$ gives
	\begin{equation}\label{eqz}
	(\mu - 1)^{2}\, x^{\!*} B^{T} B x 
	= - \big( \beta x^{\!*} A x + \alpha \beta x^{\!*} B^{T} B x \big) \mu^{2} 
	+ \big( \beta x^{\!*} A x + \beta x^{\!*} C^{T} D^{-1} C x \big) \mu.
	\end{equation}
	Define
	\[
	a = x^{\!*} A x> 0, \qquad 
	b = x^{\!*} B^{T} B x\ge 0, \qquad 
	c = x^{\!*} C^{T} D^{-1} C x\ge 0 .
	\]
	Then Eq.~\eqref{eqz} becomes
	\begin{equation}\label{eqw}
	q(\mu):=(\beta a + \alpha\beta b + b)\, \mu^{2} 
	- (\beta a + 2 b + \beta c)\, \mu 
	+ b = 0,
	\end{equation}
	and the discriminant of the quadratic polynomial $q$ in $\mu$ is
	\[
	\Delta
	= \beta^{2} (a + c)^{2} + 4 \beta b c - 4 \alpha \beta b^{2}.
	\]
	We first assume that $b = 0$. In this case~$ \Delta>0 $ and~\eqref{eqw} reduces to
	\[
	\beta a\, \mu^{2} - \beta(a + c) \mu = 0,
	\]
	whose nonzero root is
	\[
	\mu = 1 + \frac{c}{a}.
	\]
Hence, in this case, it is clear that
	\begin{align}\label{gfun}
	1+\frac{\lambda_{\min}(C^{T}D^{-1}C)}{\lambda_{\max}(A)}\le\mu\le 1+\frac{\lambda_{\max}(C^{T}D^{-1}C)}{\lambda_{\min}(A)}.
	\end{align}
	Since $\lambda_{\min}(C^{T}D^{-1}C)=0$ ,  the lower bound in \eqref{gfun} reduces to $1$, and we obtain
	\begin{equation}\label{Az1}
	1 \le \mu \le 1 + \frac{\lambda_{\max}(C^{T}D^{-1}C)}{\lambda_{\min}(A)}.
	\end{equation}
	\medskip
If $b \neq 0$, then $4\beta b c \ge 0$. Moreover, if
\[
 \beta^{2}(a + c)^{2} - 4\alpha\beta b^{2} \ge 0,
\]
then it follows that
\begin{equation}\label{cb}
\frac{\alpha}{\beta} \le \frac{(a + c)^{2}}{4b^{2}}.
\end{equation}
Observe that the inequality in \eqref{cb} is guaranteed if
\begin{equation}\label{sdd}
\frac{\alpha}{\beta} \le \frac{\left(\lambda_{\min}(A) + \lambda_{\min}\!\left(C^{T} D^{-1} C\right)\right)^{2}}{4\sigma_{\max}^{4}(B)},
\end{equation}
which in turn ensures that $\Delta \ge 0$. Given that $\lambda_{\min}(C^{T}D^{-1}C) = 0$ and $\sigma_{\max}(B) = 1$, the inequality in \eqref{sdd} simplifies to
	\[
	\frac{\alpha}{\beta} \le \frac{1}{4}\,\lambda_{\min}^{2}(A).
	\]
Consequently, the discriminant is non-negative and the roots of~\eqref{eqw} are real, namely,
\[
\mu_{1}
=
\frac{2b + \beta(a + c) - \sqrt{\Delta}}
{2(\beta a + \alpha\beta b + b)},\qquad 
\mu_{2}
=
\frac{2b + \beta(a + c) + \sqrt{\Delta}}
{2(\beta a + \alpha\beta b + b)},
\]
with $\mu_{1} < \mu_{2}$.
	 	Furthermore, since
	\[
	\mu_{1}+\mu_{2}
	=
	\frac{2b+\beta(a+c)}{\beta a+\alpha\beta b+b}
	>0,
	\qquad
	\mu_{1}\mu_{2}
	=
	\frac{b}{\beta a+\alpha\beta b+b}
	>0,
	\]
	both eigenvalues are positive.	To bound the eigenvalues, we rewrite \eqref{eqw} as
	\[
	(\mu - 1)^{2} b 
	= -(a\beta+\alpha\beta b)\mu^{2} 
	+ \beta(a+c)\mu .
	\]
	Since the left-hand side is non-negative,
	\[
	-(a\beta+\alpha\beta b)\mu^{2} + \beta(a+c)\mu \ge 0,
	\]
	which gives the upper bound
	\begin{align}\label{bound1}
	\mu \le \frac{a+c}{a+\alpha b}<1+\frac{c}{a}.
	\end{align}
So
	\begin{align}\label{mu3a}
	\mu
	<  
	1+\frac{ \lambda_{\max}(C^{T} D^{-1} C)}
	{ \lambda_{\min}(A)}.
	\end{align}

Combining Eq.~\eqref{Az1}  and Eq.~\eqref{mu3a} yields
\[
\Lambda\!\bigl(\mathcal{Q}_{\mathrm{BPT}}^{-1}\tilde{\mathcal{A}}\bigr)
\subset
\left(0
,\;
1+\frac{
	\lambda_{\max}\!\left(C^{T} D^{-1} C\right)
}{
	\lambda_{\min}(A)
}
\right].
\]
	which completes the proof.
\end{proof}
\begin{theorem}\label{th4}
	Suppose that  the matrices $A$, $B$, $C$ and $D$ satisfy the assumptions stated in Theorem~\ref{th1}.  
	Let $(\mu,\mathbf{u})$ be an eigenpair of the preconditioned matrix 
	$\mathcal{Q}_{\mathrm{BPT}}^{-1}\mathcal{\tilde{A}}$, where 
	$\mathbf{u} = (x; y; z)$. Then the eigenstructure may be characterized as follows:
	
	\medskip
	\noindent\textbf{Case I ($\mu = 1$).}  
The eigenvalue $\mu = 1$ has eigenvectors of the form
	\[
	\mathbf{u} = \left( x;\ 0;\ z \right),
	\]
	for each $x, z$ satisfying $x - \frac{1}{\alpha}C^{T}z \in \mathcal{N}(B)$.
\medskip

	\noindent\textbf{Case II ($\mu \neq 1$).}  
	For each $\mu \neq 1$, a corresponding eigenvector is given by
	\[
	\mathbf{u} = \left( x;\ \tfrac{\mu-1}{\beta\mu} Bx;\ D^{-1}Cx \right),
	\qquad x \neq 0.
	\]
\end{theorem}

\begin{proof}
	Let $\mathbf{u} = (x; y; z)$ be an eigenvector associated with the eigenvalue $\mu$, i.e.,
	\[
	\mathcal{Q}_{\mathrm{BPT}}^{-1}\mathcal{\tilde{A}} \mathbf{u}
	= \mu \mathbf{u},
	\qquad\text{equivalently}\qquad
	\mathcal{\tilde{A}}\mathbf{u}
	= \mu\,\mathcal{Q}_{\mathrm{BPT}}\mathbf{u}.
	\]
	This relation expands to
	\begin{align}\label{eqss}
	\begin{pmatrix}
	A & B^{T} & C^{T} \\
	B & 0 & 0 \\
	C & 0 & -D
	\end{pmatrix}
	\begin{pmatrix}
	x \\ y \\ z
	\end{pmatrix}
	=
	\mu\!
	\begin{pmatrix}
	A + \alpha B^{T}B & B^{T} & 0 \\
	B & -\beta I & 0 \\
	C & 0 & -D
	\end{pmatrix}
	\begin{pmatrix}
	x \\ y \\ z
	\end{pmatrix},
	\end{align}
	where $x \in \mathbb{R}^{n}$, $y \in \mathbb{R}^{m}$, and $z \in \mathbb{R}^{p}$.  
	Eq.~\eqref{eqss} is equivalent to the system
	\begin{align}
	Ax + B^{T}y + C^{T}z &= \mu\bigl(Ax + \alpha B^{T}Bx + B^{T}y\bigr),  \label{eqd3} \\
	Bx &= \mu(Bx - \beta y), \label{eqx1} \\
	(1-\mu)(Cx - Dz) &= 0. \label{eqm}
	\end{align}
	
	\medskip\noindent\textbf{Case I: $\mu = 1$.}  
	Eq.~\eqref{eqx1} reduces to $y = 0$.  
	Substituting $y = 0$ into~\eqref{eqd3} yields
	\begin{equation}\label{eqr}
	C^{T}z = \alpha B^{T}Bx.
	\end{equation}
	This is analogous to Case~1 in Theorem~\ref{th3}, and the eigenvectors take the form
	\[
		\mathbf{u} = \left( x;\, 0;\, z \right),
		\]
	corresponding to $\mu = 1$, for each $x, z$ satisfying $x - \frac{1}{\alpha}C^{T}z \in \mathcal{N}(B)$.
	
	\medskip\noindent\textbf{Case II: $\mu \neq 1$.}  
	In this case, Eqs.~\eqref{eqd3}–\eqref{eqm} become
	\begin{align}
	Ax + B^{T}y + C^{T}z &= \mu\bigl(Ax + \alpha B^{T}Bx + B^{T}y\bigr), \label{eqdd} \\
	y &= \frac{\mu - 1}{\beta\mu}\,Bx, \label{eqxx} \\
	z &= D^{-1}Cx. \label{eqmm}
	\end{align}
	Substituting~\eqref{eqxx} and~\eqref{eqmm} into~\eqref{eqdd} we obtain
	\begin{equation}\label{ert1}
	\Bigl((1-\mu)A 
	- \frac{(1+\alpha\beta)\mu^{2} - 2\mu + 1}{\beta\mu} B^{T}B
	- C^{T}D^{-1}C
	\Bigr)x = 0.
	\end{equation}
	Since $\mathcal{Q}_{\mathrm{BPT}}^{-1}\mathcal{\tilde A}$ and $\mathcal{G}_{\alpha,\beta}$ share the same eigenvectors, and by Theorem~\ref{th1}, it follows that $x \neq 0$.
	Without loss of generality, we may assume that $\|x\|_{2} = 1$.  
	Premultiplying~\eqref{ert1} by $x^{\ast}$ yields the  quadratic equation
	\begin{equation}\label{tqad}
	\mu^{2}
	- \frac{\beta a + 2b + \beta c}{\beta a + \alpha\beta b + b}\,\mu
	+ \frac{b}{\beta a + \alpha\beta b + b}
	= 0,
	\end{equation}
	where $a$, $b$, and $c$ are as defined in~\eqref{aandb}.  	
	Thus, each root $\mu$ of~\eqref{tqad} is an eigenvalue of 
	$\mathcal{Q}_{\mathrm{BPT}}^{-1}\mathcal{\tilde{A}}$, and the corresponding eigenvector is
	\[
	\mathbf{u} 
	= \left( x;\ \tfrac{\mu - 1}{\beta\mu}Bx;\ D^{-1}Cx \right).
	\]
	The proof is therefore concluded. 
\end{proof}

\section{A strategy for preconditioner parameters selection}\label{sec4}
In the implementation of the BPT preconditioner, suitable choices of the parameters $\alpha$ and $\beta$ play a crucial role in achieving efficient performance. 
Considering that the relation
\[
\Phi(\alpha, \beta)
= \|\mathcal{Q}_{\mathrm{BPT}}- \tilde{\mathcal{A}}\|_F^2
= (\alpha^2 +\beta^2)m + \| C \|_F^2 
\]
holds, it is evident that adopting relatively small values for $\alpha$ and $\beta$ generally improves the approximation quality. 

From Eq.~\eqref{Pmatrix}, it follows immediately that the preconditioned matrix 
\(\mathcal{Q}_{\mathrm{BPT}}^{-1} \tilde{\mathcal{A}}\) admits an eigenvalue equal to \(1\) with algebraic multiplicity at least \(p\). The remaining eigenvalues are precisely those of the matrix
\[
\Gamma =
\begin{pmatrix}
I + \bigl( C^{T} D^{-1} C - \alpha B^{T} B \bigr) M^{-1} &
\displaystyle \frac{1}{\beta} \bigl( C^{T} D^{-1} C - \alpha B^{T} B \bigr) M^{-1} B^{T} \\[8pt]
B M^{-1} &
\displaystyle \frac{1}{\beta} B M^{-1} B^{T}
\end{pmatrix}.
\]
The efficiency of a preconditioner is intrinsically related to the spectral properties of the preconditioned matrix $\mathcal{Q}_{\mathrm{BPT}}^{-1} \tilde{\mathcal{A}}$. 
In particular, when the eigenvalues of this matrix are clustered near unity, the convergence rate of the associated iterative method is typically enhanced. According to~\cite{BenziDe, Zhuw}, the parameters $\alpha$ and $\beta$ can be determined by solving the minimization problem
\begin{align}\label{mintr}
\min_{\alpha,\beta}\operatorname{tr}\!\left( \mathcal{Q}_{\mathrm{BPT}}^{-1} \tilde{\mathcal{A}}-I \right).
\end{align}
Based on~\eqref{mintr}, suitable values for $\alpha$ and $\beta$ can be obtained by imposing the approximate conditions
\begin{align}\label{alphaest}
I_n + \left( C^{T} D^{-1} C - \alpha B^{T} B \right) M^{-1} &\approx I_n, \\[4pt]
\frac{1}{\beta} B M^{-1} B^{T} &\approx (1 - \theta) I_m, \label{betaest}
\end{align}
where $\theta > 0$ is a small parameter. Under these approximations, the BPT preconditioner is expected to improve computational performance.
From~\eqref{alphaest}, the parameter $\alpha$ can be chosen to minimize 
\[
\operatorname{tr}\!\left( C^{T} D^{-1} C - \alpha B^{T} B \right).
\]
This minimization leads to the estimate
\begin{align}\label{alphaes1}
\alpha_{\mathrm{est}}^{\mathrm{BPT}} 
\equiv \frac{\operatorname{tr}\!\left( C^{T} \hat{D}^{-1} C \right)}
{\operatorname{tr}\!\left( B^{T} B \right)}=\frac{\operatorname{tr}\!\left( C^{T}  \hat{D}^{-1} C \right)}
{\operatorname{tr}\!\left( BB^{T}  \right)}=\frac{\operatorname{tr}\!\left( C^{T}  \hat{D}^{-1} C \right)}
{m},
\end{align}
where $\hat{D}=\diag(D)$.

Similarly, from~\eqref{betaest} and applying the Sherman--Morrison--Woodbury formula~\cite{Horn2013}, we obtain
\[
B A^{-1} B^{T} \approx 
\beta (1 - \theta) 
\left( I_m + \bigl( \alpha + \tfrac{1}{\beta} \bigr) B A^{-1} B^{T} \right).
\]
Simplifying this relation the following estimation yields
\begin{align}\label{betaes1}
\beta_{\mathrm{est}}^{\mathrm{BPT}}
\equiv
\frac{ \theta\, \operatorname{tr}\!\left( B D_A^{-1} B^{T} \right) }
{ (1 - \theta) \left( m + \alpha\, \operatorname{tr}\!\left( B D_A^{-1} B^{T} \right) \right) },
\end{align}
where $D_A=\diag(A)$.


\section{Numerical Experiments}\label{sec5}

In this section, we examine the practical efficiency of the proposed BPT preconditioner for solving double saddle point systems through a collection of numerical tests. All computations were carried out in \textsc{Matlab}~(R2020a) on a Windows~10 platform equipped with an Intel Core~i5 processor running at 2.6~GHz and 8~GB of RAM.
 
 In all the tests, the initial vector is a zero vector and
 the right-hand side vector $ \mathbf{d} $ is chosen such that the exact solution of~\eqref{org} is the vector that
 has all its components equal to one.

We compare the proposed BPT preconditioner with the TPBT~\cite{Zhuw} and BUT~\cite{Zhang25}, preconditioners, defined in~\eqref{TPBT} and~\eqref{PBUT}, respectively. To ensure that the tested preconditioners are implemented efficiently and that the corresponding preconditioned FGMRES solvers achieve rapid convergence, it is crucial to select their parameters appropriately. For the TPBT and BUT preconditioners, the parameters $\alpha_{\text{est}}$ and $\beta_{\text{est}}$ are chosen according to the algebraic estimation strategy derived in Section~\ref{sec4} as follows.
\begin{eqnarray*}
\alpha_{\mathrm{est}}^{\mathrm{TPBT}} &=& \frac{\operatorname{tr}\left(B D_A^{-1} B^T\right)}{m}, \hspace{1.6cm}
\quad \beta_{\mathrm{est}}^{\mathrm{TPBT}} = \operatorname{tr}(D_A^{-1}), \\
\alpha_{\mathrm{est}}^{\mathrm{BUT}} &=& \frac{\varepsilon}{m(1 - \varepsilon)} \operatorname{tr}\left(B D_A^{-1} B^T\right), 
\quad \beta_{\mathrm{est}}^{\mathrm{BUT}}= \frac{1}{m} \operatorname{tr}(CD_{S_{A}}^{-1}C^T).
\end{eqnarray*}
Here, $D_A$ and $D_{S_A}$ represent diagonal matrices formed by extracting the main diagonal entries of $A$ and $S_A=A+\frac{1}{\beta}B^{T}B$, respectively, while $\epsilon$ denotes a small positive parameter. In the numerical experiments presented in Tables~\ref{tab.1} and~\ref{tab.2}, the parameter $\theta$ of the BPT method was set to $0.001$, while for the BUT method, the parameter $\varepsilon$ was chosen as $0.0001$, in accordance with~\cite{Zhang25}.

 The preconditioned  flexible GMRES (FGMRES) method is employed, with a maximum of 1000 iterations allowed and a relative residual tolerance of $10^{-10}$. The iteration is terminated when
\[
\text{RES}=\frac{\|\mathbf{d} - \mathcal{\tilde{A}}\mathbf{u}_k\|_2}{\|\mathbf{d}\|_2} \leq 10^{-10}.
\]
To assess the accuracy of the computed solutions, we report the relative error (ERR) of the numerical approximations, defined as
\[
\text{ERR} = \frac{\|\mathbf{u}_k - \mathbf{u}_*\|_2}{\|\mathbf{u}_*\|_2},
\]
where $\mathbf{u}_k$ denotes the approximate solution at the $k$-th iteration and $\mathbf{u}_*$ represents the exact solution of the linear system.

The linear subsystems involved in the flexible
GMRES method were solved inexactly by the CG method, where the preconditioners
were chosen as the sparse Cholesky factorizations with AMD reordering. The inner stopping tolerance was set to $10^{-3}$, and the maximum number of iterations for CG was limited to 100.
 All reported CPU times (in seconds) and iteration counts correspond to the average over three independent runs.
 
\begin{example}\label{ex1}\rm
	 (Off-state problems)~We consider the energy functional \eqref{eqq1} with $\hat{\alpha}=0.5\alpha_{c}$ and $\hat{\beta}=0.5$, where $\alpha_{c}\approx 2.721$ is known as the critical switching value. Moreover, we choose the mesh step-size $h=1/(l+1)$ for the uniform piecewise linear finite element scheme, where $l+1$ is the mesh number along each direction. Furthermore, we impose the unit constraint conditions $u_{j}^{2}+v_{j}^{2}+w_{j}^{2}=1$ for $j=1,2,\ldots,l$, where $u_{j}$, $v_{j}$, $w_{j}$ represent the components of solution vectors $u$, $v$, $w$, respectively. Then, the double saddle point coefficient matrix $\mathcal{\tilde{A}}$ of \eqref{org} is obtained with $n=3l$ and $m=p=l$.
\end{example}
 
 \begin{table}[htbp]
 	\centering
 	\caption{Spectral properties of the matrix \(A-C^{T}D^{-1}C\) for different sizes.}
 	\label{tab:eigenvalues_AB}
 	\begin{tabular}{cccc}
 		\toprule
 		Matrix & Problem Size &  \(\boldsymbol{\lambda_{\min}}\) & \(\boldsymbol{\lambda_{\max}}\) \\
 		\midrule
 		 & \(635\)  & \(6.44 \times 10^{-2}\) & \(5.11 \times 10^2\)  \\
 		 & \(1275\)  & \(3.21 \times 10^{-2}\) & \(1.02 \times 10^3\)  \\
 		\(A-C^{T}D^{-1}C\) & \(2555\) & \(1.60 \times 10^{-2}\) & \(2.04 \times 10^3\)  \\
 		 & \(5115\) & \(8.04 \times 10^{-3}\) & \(4.09 \times 10^3\)  \\
 		 & \(10235\) & \(4.02 \times 10^{-3}\) & \(8.19 \times 10^3\)  \\
 		\bottomrule
 	\end{tabular}
 \end{table}

\begin{figure}[htbp]
	\centering
	
	\begin{subfigure}{0.48\textwidth}
		\centering
		\includegraphics[width=\textwidth]{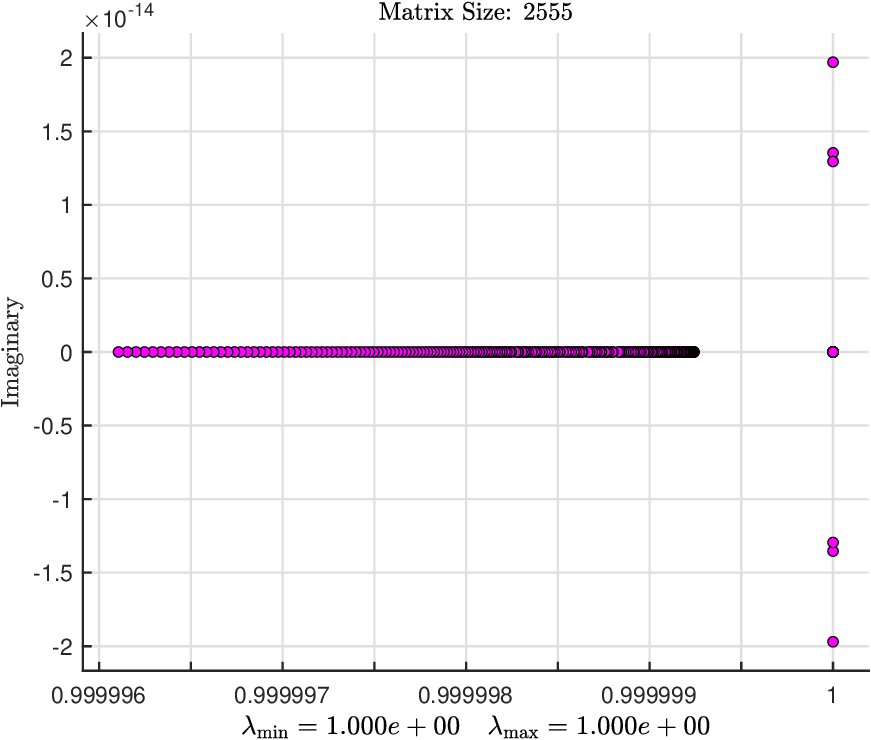}  
		\label{fig:left}
	\end{subfigure}
	\hfill 
	\begin{subfigure}{0.48\textwidth}
		\centering
		\includegraphics[width=\textwidth]{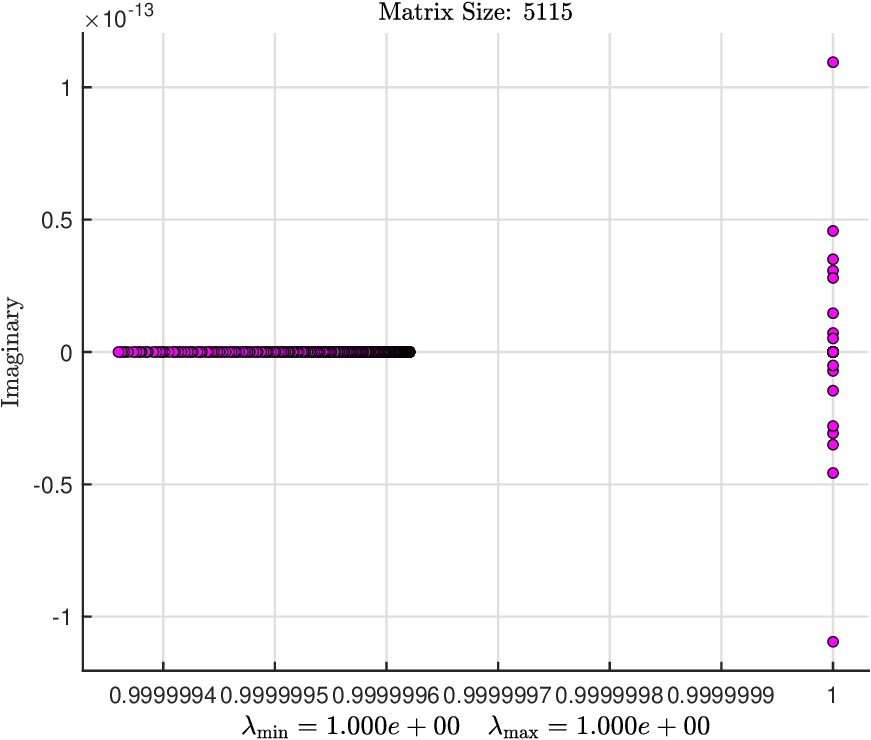} 
		\label{fig:right}
	\end{subfigure}
	
	\caption{Distribution of eigenvalues of the preconditioned matrix $ \mathcal{Q}_{\mathrm{BPT}}^{-1}\mathcal{\tilde{A}} $ based on the parameters of Theorem~\ref{th3}, for Example~\ref{ex1}.}
	\label{fig:main}
\end{figure}

The numerical results in Table~\ref{tab:eigenvalues_AB} confirm the theoretical 
findings of Theorem~\ref{th1} and the subsequent remark, 
demonstrating that our method converges unconditionally for any $\alpha > 0$ 
and $\beta > 0$.

To validate the theoretical results of Theorem~\ref{th3}, Figure~\ref{fig:main} 
is illustrated. The parameters are selected based on the following relation from this theorem
\begin{align}\label{alphabeta1}
\frac{\alpha}{\beta} \le \frac{1}{4}\,\lambda_{\min}^{2}(A),
\end{align}
specifically, $\alpha = 1.4862 \times 10^{-3}$ and $\beta = 1.4862 \times 10^{4}$ 
are chosen to satisfy Equation~\eqref{alphabeta1}. The right subfigure corresponds 
to Example~\ref{ex1} with a size of 5115, where $\lambda_{\min}^2(A)/4 = 2.3224 \times 10^{-5}$, 
while the left subfigure corresponds to a size of 2555, yielding 
$\lambda_{\min}^2(A)/4 = 9.2896 \times 10^{-5}$. It is readily apparent that choosing 
the parameters according to Theorem~\ref{th3} effectively clusters the eigenvalues 
around the point 1.

For a test problem of size 2555, the eigenvalue distributions of the various preconditioned matrices, along with their corresponding estimated parameters, are depicted in Figure~\ref{fig2}. It is evident that the preconditioned matrix $\mathcal{Q}_{\text{BPT}}^{-1}\mathcal{\tilde{A}}$ yields a much tighter and more favorable clustering behavior. It is necessary to mention that the imaginary parts of the preconditioned matrix for the BPT preconditioner are due to the round-off error.

Figure~\ref{fig:fgmres_plots} examines the robustness of the proposed preconditioner with
respect to the parameter $\hat{\alpha}$, which, according to the discussion
in~\cite{Ramage2013}, is the main control parameter. The parameter $\hat{\beta}$ is
kept fixed at $\hat{\beta} = 0.5$. The values of $\hat{\alpha}$ used in these tests
are
\[
\hat{\alpha}(k) = 0.5 \, k \, \alpha_c, \qquad k = 1,\dots,8,
\]
with $\alpha_c = 2.721$. These values are physically meaningful and
match those employed in the reference tests of~\cite{Ramage2013}. The results
for a medium-sized problem and a large-sized problem, with
$n = 20,475$ and $n = 163,835$ degrees of freedom, respectively, are
shown in Figure~\ref{fig:fgmres_plots}; the trends for the other problem sizes are similar. The same conclusion holds for even larger problem sizes as well.
From these results, it can be observed that the preconditioner performance is essentially independent of $\hat{\alpha}$.
 
 \begin{figure}[!htp]
 	\centering
 	\includegraphics[scale=0.90]{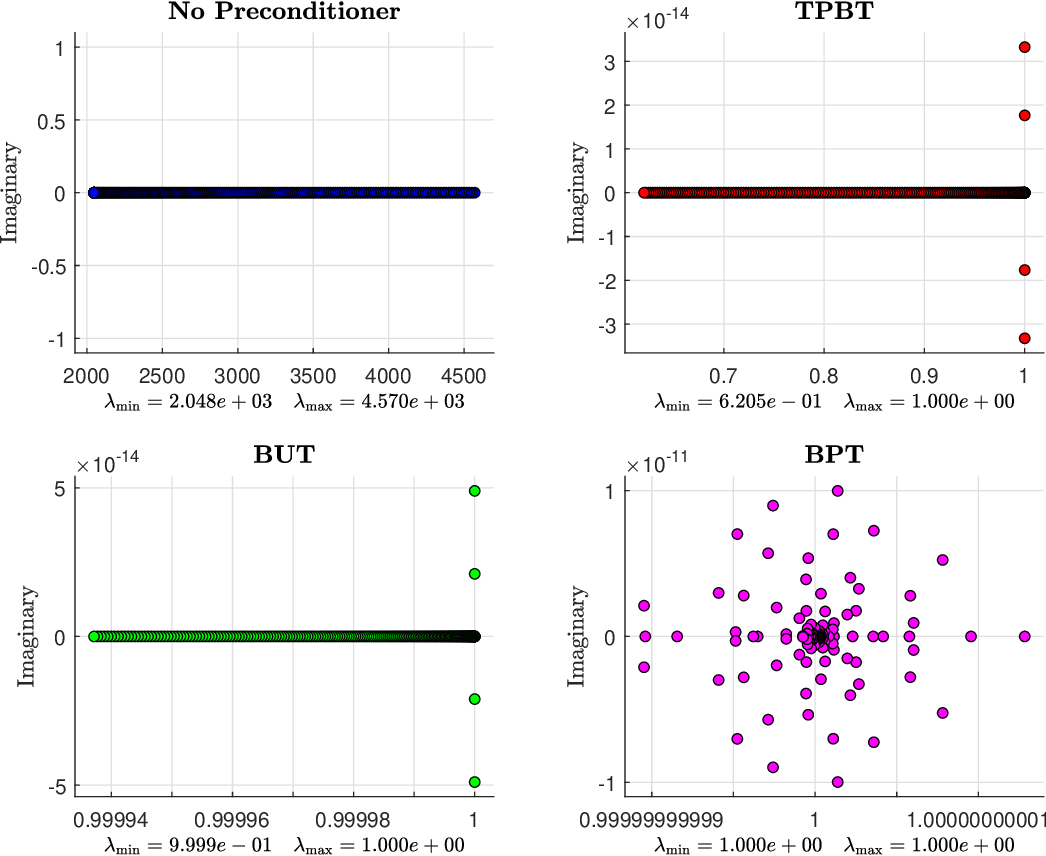}
 	\caption{Eigenvalue distribution of the preconditioned matrix for the estimated parameters, for Example~\ref{ex1}.}
 	\label{fig2}
 \end{figure}

\begin{figure}[htbp]
	\centering
	
	\begin{tikzpicture}
	\begin{axis}[
	width=0.48\textwidth,
	title={Matrix Size: 20475},
	xlabel={$k$},
	ylabel={Number of Iterations},
	xmin=1, xmax=8,
	ymin=0, ymax=4,
	xtick={1,2,3,4,5,6,7,8},
	ytick={0,1,2,3,4},
	grid=major,
	grid style={dashed,gray!30}
	]
	\addplot[
	color=blue,
	mark=*,
	thick
	]
	coordinates {
		(1,2)(2,2)(3,2)(4,2)(5,2)(6,2)(7,2)(8,2)
	};
	\end{axis}
	\end{tikzpicture}
	\hfill 
	\begin{tikzpicture}
	\begin{axis}[
	width=0.48\textwidth,
	title={Matrix Size: 163835},
	xlabel={$k$},
	xmin=1, xmax=8,
	ymin=0, ymax=4,
	xtick={1,2,3,4,5,6,7,8},
	ytick={0,1,2,3,4},
	grid=major,
	grid style={dashed,gray!30}
	]
	\addplot[
	color=red,
	mark=square*,
	thick
	]
	coordinates {
		(1,2)(2,2)(3,2)(4,2)(5,2)(6,2)(7,2)(8,2)
	};
	\end{axis}
	\end{tikzpicture}
	
	\caption{Number of iterations for FGMRES vs $k$ for $\alpha_{\mathrm{est}}$, $\beta_{\mathrm{est}}$, $\theta=0.001$}
	\label{fig:fgmres_plots}
\end{figure}
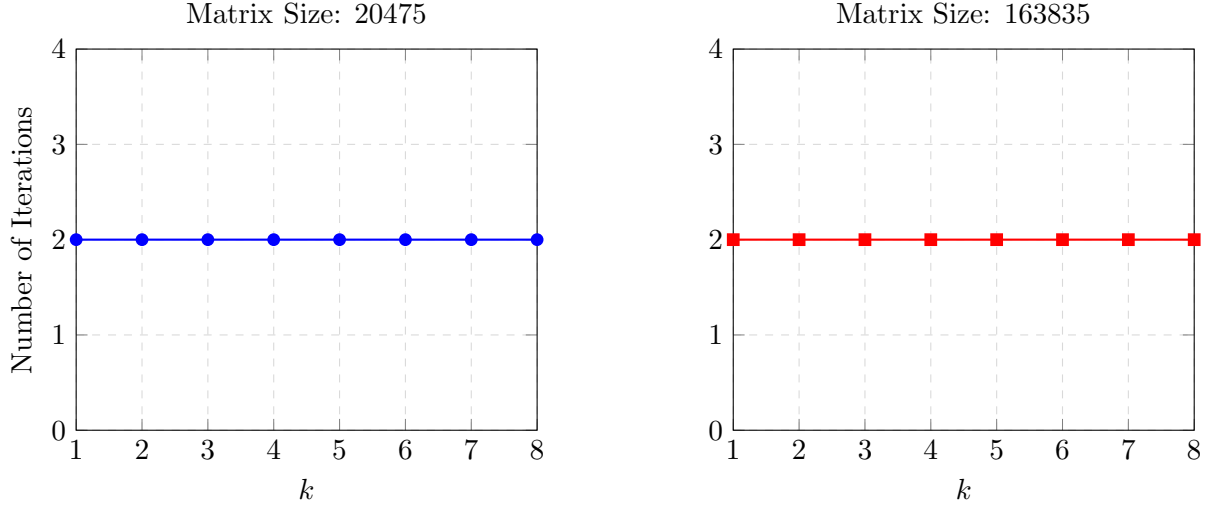

In Table~\ref{tab.11}, we report the numerical results for the  PBT preconditioner, for Example~\ref{ex1}. The parameters $\alpha_{\mathrm{est}}$ and $\beta_{\mathrm{est}}$ are chosen according to relations~\eqref{alphaes1} and \eqref{betaes1}, respectively.
To demonstrate that the proposed method exhibits little sensitivity to the parameter $\theta$ associated with relation \eqref{betaes1}, we report the results for $\theta = 0.1$, $0.01$, and $0.0001$; nevertheless, for the main comparative experiments, the BPT method was consistently employed with $\theta = 0.001$, as discussed earlier.

The numerical results for Example~\ref{ex1} are presented in Table~\ref{tab.1}. These results are obtained with $\alpha_{\mathrm{est}}$, $\beta_{\mathrm{est}}$, and with $\theta = 0.001$.
 We observe that the BPT preconditioner consistently outperforms the TPBT and BUT preconditioners in terms of both iteration count and CPU time across all tested mesh sizes. In particular, the FGMRES method with the BPT preconditioner requires significantly fewer iterations than the other methods, and the iteration counts remain constant.
Additionally, the proposed preconditioning strategy provides solutions of markedly higher fidelity than those obtained via the alternative preconditioners. A striking observation is the progressive loss of accuracy in the competing approaches as the mesh is refined and the system size grows, whereas the BPT preconditioned solver preserves its precision throughout. This resilience to scale-up further establishes the superior robustness and computational efficacy of the method herein advanced.
\begin{table}[htbp]
	\centering
	\caption{FGMRES performance with different values of $\theta$ for Example~\ref{ex1}.}
	\label{tab.11}
	\setlength{\tabcolsep}{5pt}
	\renewcommand{\arraystretch}{1.1}
	\begin{tabular}{lcccccccccccc}
		\toprule
		\multirow{2}{*}{Size} &
		\multicolumn{4}{c}{$ \theta=0.1 $} & 
		\multicolumn{4}{c}{$ \theta=0.01 $} & 
		\multicolumn{4}{c}{$ \theta=0.0001 $} \\
		\cmidrule(lr){2-5}\cmidrule(lr){6-9}\cmidrule(lr){10-13}
		& IT & CPU & RES & ERR & IT & CPU & RES & ERR & IT & CPU & RES & ERR \\
		\midrule
		
		20475 & 2 & 0.09 & 8.1e-11 & 5.1e-05 & 2 & 0.09 & 1.7e-12 & 5.5e-07 & 2 & 0.09 &3.8e-13 & 2.4e-09 \\
		40955 & 2 & 0.13 & 2.0e-11 & 3.6e-05 & 2 & 0.11 & 4.9e-13 & 3.9e-07 & 2 & 0.11 & 1.0e-13 & 3.9e-09 \\
		81915 & 2 & 0.17 & 4.9e-12 & 2.6e-05 & 2 & 0.16 & 8.9e-14 & 2.8e-07 & 2 & 0.16 & 2.5e-13 & 2.8e-08 \\
		163835 &2 & 0.27 & 1.2e-12 & 1.8e-05 & 2 & 0.25 & 4.3e-14 & 1.9e-07 & 2 & 0.26 & 2.7e-13 & 5.9e-08 \\
		327675 & 2 & 0.52 &3.1e-13 & 1.3e-05 & 2 & 0.47 & 5.0e-14 & 1.4e-07 & 2 & 0.49 & 3.8e-13 & 2.6e-07 \\
		655355 & 2 & 0.97 & 1.1e-13 & 9.1e-06 &2& 0.96 & 7.5e-14 & 1.0e-07 & 2 & 0.96 & 5.7e-13 & 1.5e-06 \\
		1310715 & 2 & 1.83 & 1.1e-13 & 6.4e-06 &2 & 1.80 & 1.0e-13 & 1.5e-07 & 2 & 1.80 & 9.5e-12 & 4.4e-06 \\
		2621435 & 2 & 3.47 &1.5e-13 & 4.6e-06 &2 & 3.48 & 1.5e-13 & 5.3e-07 & 2 & 4.46 & 2.7e-11 & 1.3e-04 \\
		\bottomrule
	\end{tabular}
\end{table}

\begin{table}[htbp]
	\centering
	\caption{Comparison of FGMRES performance with different preconditioners for Example~\ref{ex1}.}
	\label{tab.1}
	\setlength{\tabcolsep}{5pt}
	\renewcommand{\arraystretch}{1.1}
	\begin{tabular}{lcccccccccccc}
		\toprule
		\multirow{2}{*}{Size} &
		\multicolumn{4}{c}{TPBT} & 
		\multicolumn{4}{c}{BUT} & 
		\multicolumn{4}{c}{BPT} \\
		\cmidrule(lr){2-5}\cmidrule(lr){6-9}\cmidrule(lr){10-13}
		& IT & CPU & RES & ERR & IT & CPU & RES & ERR & IT & CPU & RES & ERR \\
		\midrule

		  20475& 5 & 0.14 & 9.1e-11 & 7.3e-05 & 4 & 0.12 & 4.5e-12 & 4.5e-09 & 2 & 0.09 & 4.5e-13 & 6.9e-09 \\
		 40955 & 5 & 0.19 & 2.3e-11 & 5.2e-05 & 4 & 0.15 & 1.6e-12 & 4.6e-09 & 2 & 0.10 & 6.5e-14 & 4.5e-09 \\
			81915 & 4 & 0.25 & 2.7e-11 & 1.6e-04 & 4 & 0.23 & 5.6e-13 & 4.5e-09 & 2 & 0.16 & 3.5e-14 & 3.1e-09 \\
	 163835 & 3 & 0.36 & 5.5e-11 & 2.1e-04 & 3 & 0.33 & 6.8e-11 & 1.7e-06 & 2 & 0.26 & 4.5e-14 & 2.9e-09 \\
		327675 & 3 & 0.72 &1.9e-11 & 1.5e-04 & 3 & 0.63 & 5.1e-11 & 4.9e-06 & 2 & 0.50 & 7.8e-14 & 5.5e-08 \\
		655355 & 3 & 1.38 & 6.7e-12 & 1.1e-04 & 3 & 1.21 & 6.3e-11 & 1.3e-05 & 2 & 0.96 & 8.4e-14 & 1.2e-07 \\
		1310715 & 3 & 2.64 & 2.4e-12 & 7.5e-05 & 3 & 2.35 & 9.1e-12 & 7.2e-06 & 2 & 1.83 & 1.2e-13 & 3.8e-07 \\
		2621435 & 3 & 5.11 &8.5e-13 & 5.3e-05 & 4 & 5.51 & 1.6e-12 & 7.2e-06 & 2 & 3.46 & 2.2e-13 & 2.4e-07 \\
		\bottomrule
	\end{tabular}
\end{table}
\begin{example}\label{ex2}\rm
	(On-state problems) We consider the energy functional \eqref{eqq1} with $\hat{\alpha}=1.5\alpha_{c}$ and $\hat{\beta}=0.5$, where $\alpha_c\approx 2.721$. Similarly, using uniform piecewise linear finite element scheme with the mesh step size $h=1/(l+1)$ and imposing the constraint conditions $u_{j}^{2}+v_{j}^{2}+w_{j}^{2}=1$ for $j=1,\,2,\,\ldots,\,l$, respectively, the double saddle point coefficient matrix $\mathcal{\tilde{A}}$ of \eqref{org} is obtained with $n=3l$ and $m=p=l$.
\end{example}

Figure~\ref{fig3} illustrates the eigenvalue distributions of the preconditioned matrices, along with their estimated parameters, for a test problem of size $2555$. As is clearly observed, the preconditioned matrix $\mathcal{Q}_{\text{BPT}}^{-1}\mathcal{\tilde{A}}$ exhibits a significantly more compact and favorable eigenvalue clustering.

The computational results  are summarized in Table~\ref{tab.2}. It is evident that the BPT preconditioner consistently surpasses both the TPBT and BUT preconditioners with respect to iteration count and CPU time across all mesh sizes under consideration. Notably, the FGMRES solver equipped with the BPT preconditioner requires substantially fewer iterations than its counterparts, while maintaining nearly constant iteration counts independent of problem scale.

Furthermore, the proposed BPT preconditioning framework delivers solutions of significantly higher accuracy compared to the alternative methods. In contrast to the TPBT and BUT approaches, which exhibit a clear degradation in solution precision as the mesh is refined and the system dimension increases, the BPT preconditioned solver sustains excellent accuracy throughout. This exceptional scalability and robustness firmly establish the superior performance and reliability of the BPT preconditioner.

\begin{figure}[!htp]
	\centering
	\includegraphics[scale=0.90]{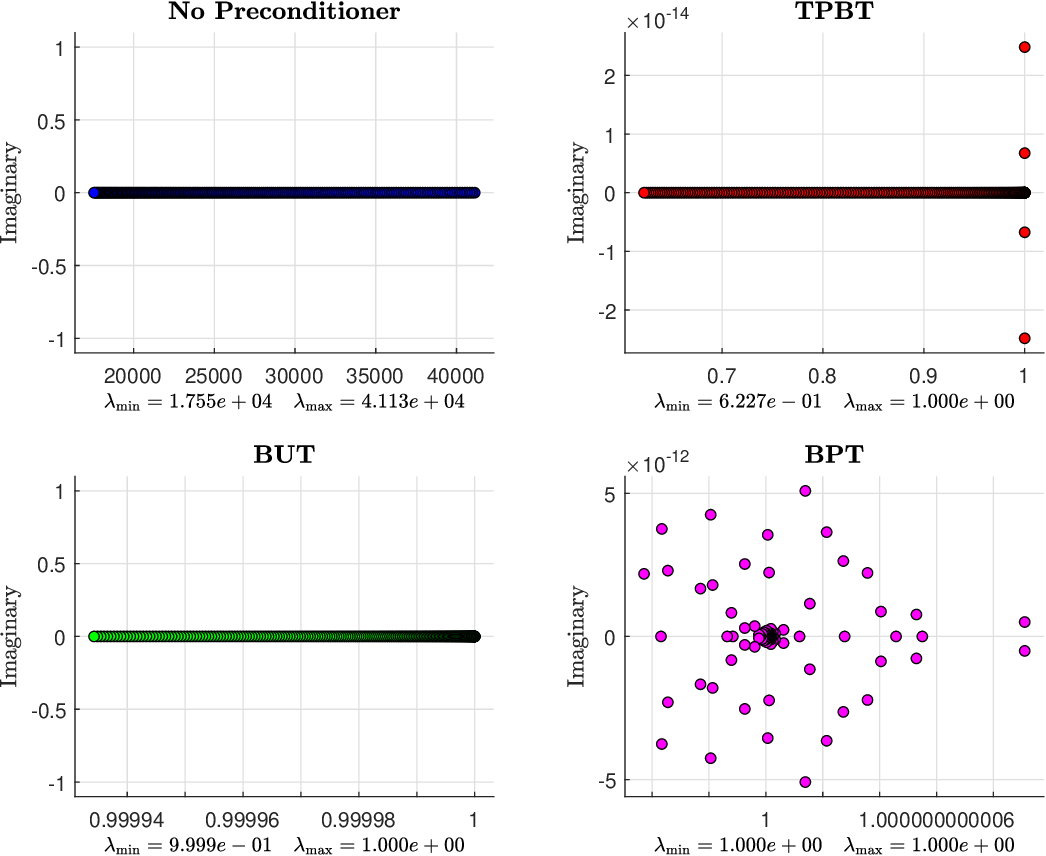}
	\caption{Eigenvalue distribution of the preconditioned matrix for the estimated parameters, for Example~\ref{ex2}.}
	\label{fig3}
\end{figure}

\begin{table}[htbp]
	\centering
	\caption{Comparison of FGMRES performance with different preconditioners for Example~\ref{ex2}.}
	\label{tab.2}
	\setlength{\tabcolsep}{5pt}
	\renewcommand{\arraystretch}{1.1}
	\begin{tabular}{lcccccccccccc}
		\toprule
		\multirow{2}{*}{Size} &
		\multicolumn{4}{c}{TPBT} & 
		\multicolumn{4}{c}{BUT} & 
		\multicolumn{4}{c}{BPT} \\
		\cmidrule(lr){2-5}\cmidrule(lr){6-9}\cmidrule(lr){10-13}
		& IT & CPU & RES & ERR & IT & CPU & RES & ERR & IT & CPU & RES & ERR \\
		\midrule
		
		20475& 5 & 0.14 & 6.7e-11 & 2.3e-04 & 5 & 0.14 & 4.1e-11 & 4.8e-08 & 2 & 0.09 & 8.1e-12 & 1.4e-07 \\
		
		40955 & 5 & 0.19 & 1.9e-11 & 1.6e-04 & 5 & 0.18 & 1.4e-11 & 4.8e-08 & 2 & 0.11 & 7.3e-13 & 3.7e-08 \\
		
		81915 & 5 & 0.31 & 5.8e-11 & 1.1e-04 & 5 & 0.29 & 5.1e-12 & 4.8e-08 & 2 & 0.16 & 8.4e-14 & 1.4e-08 \\
		
		163835 & 5 & 0.50 & 1.9e-12 & 8.1e-05 & 5 & 0.48 & 1.8e-12 & 4.8e-08 & 2 & 0.25 & 6.3e-14 & 7.9e-09 \\
		
		327675 & 4 & 0.84 &4.7e-11 & 9.9e-05 & 4 & 0.75 & 4.8e-11 & 3.0e-06 & 2 & 0.49 & 9.0e-14 & 1.7e-08 \\
		
		655355 & 4 & 1.68 & 1.7e-11 & 7.0e-05 & 4 & 1.49 & 1.7e-11 & 3.0e-06 & 2 & 0.91 & 1.2e-13 & 1.1e-07 \\
		
		1310715 & 4 & 3.15 &6.0e-12 & 5.0e-05 & 4 & 2.87 & 1.3e-11 & 4.5e-06 & 2 & 1.78 & 2.4e-13 & 2.7e-06 \\
		
		2621435 & 3 & 5.08 &7.0e-11 &1.1e-04 & 3 & 4.55 & 1.0e-10 & 2.2e-04 & 2 & 3.50 & 3.1e-13 & 4.7e-06 \\
		\bottomrule
	\end{tabular}
\end{table}

\section{Conclusion}\label{sec6}

In this work, we introduced an efficient preconditioner for solving the double saddle point linear systems that arise in the modeling of liquid crystal director fields, and we analyzed its convergence properties. A detailed spectral analysis of the preconditioned matrix was carried out, including the characterization of eigenvalues, their multiplicities, and their structural properties. Furthermore, we derived estimates for the parameters $\alpha$ and $\beta$ that enable the preconditioner to provide a close approximation to the original coefficient matrix.
The proposed preconditioner was incorporated into a flexible GMRES framework to enhance convergence, and its performance was evaluated against several widely used preconditioners. The numerical experiments demonstrate that the proposed approach consistently outperforms existing methods, both in iteration counts and computational efficiency.

\section*{Conflict of Interest}
The authors declared that they have no conflict of interest.

\section*{Funding statement}
Not applicable.

\section*{Author Contributions} 

All three authors contributed equally in all sections of the paper.

\section*{Acknowledgement}

We would like to thank Alison Ramage  for providing the test problems used in the numerical experiments.

\end{document}